\documentclass[reqno]{amsart}
\usepackage{amsmath, amsfonts, amssymb, latexsym}
\usepackage[dvips]{graphicx}
\graphicspath{{Img/}}

\newtheorem{theorem}{Theorem}

\theoremstyle{plain}

\newtheorem{corollary}{Corollary}

\newtheorem{definition}{Definition}
\newtheorem{example}{Example}

\newtheorem{problem}{Problem}[]
\newtheorem{proposition}{Proposition}
\newtheorem{remark}{Remark}

\numberwithin{equation}{section}

\newcommand{\eq}{\hspace*{-2mm}&=&\hspace*{-2mm}}

\begin{document}
\title[A plasticity principle of convex quadrilaterals]{A plasticity principle of convex quadrilaterals on a convex surface of bounded specific curvature}
\author{Anastasios Zachos}

\address{University of Patras, Department of Mathematics, GR-26500 Rion, Greece}

\email{azachos@gmail.com,zaxos@master.math.upatras.gr}


\begin{abstract}
We derive the plasticity equations for convex quadrilaterals on a
complete convex surface with bounded specific curvature and prove
a plasticity principle which states that: \textit{Given four
shortest arcs which meet at the weighted Fermat-Torricelli point
their endpoints form a convex quadrilateral and the weighted
Fermat-Torricelli point belongs to the interior of this convex
quadrilateral, an increase of the weight corresponding to a
shortest arc causes a decrease of the two weights that correspond
to the two neighboring shortest arcs and an increase of the weight
corresponding to the opposite shortest arc} by solving the inverse
weighted Fermat-Torricelli problem for quadrilaterals on a convex
surface of bounded specific curvature. Furthermore, we show a
connection between the plasticity of convex quadrilaterals on a
complete convex surface with bounded specific curvature with the
plasticity of some generalized convex quadrilaterals on a manifold
which is certainly composed by triangles. We also study some cases
of symmetrization of weighted convex quadrilaterals by introducing
a new symmetrization technique which transforms some classes of
weighted geodesic convex quadrilaterals on a convex surface to
parallelograms in the tangent plane at the weighted
Fermat-Torricelli point of the corresponding quadrilateral.
\end{abstract}

\maketitle

\section{introduction}
\label{sec:intro}

We start with the definitions of the specific curvature of a
domain $W$ and the Gaussian curvature at a point $P$ on a convex
surface which is considered as the whole boundary of a convex body
in $\mathbb{R}^{3}$, and mention some fundamental results on a
complete convex surface of bounded specific curvature from the
book of A.D.~Alexandrov (see \cite{Alexandrov:06}).

\begin{definition}[{see \cite[pp.~365--366]{Alexandrov:06}}]\rm
Any domain $W$ on a convex surface has some curvature
$\omega(W)=\iint_{W} K(P)\,dS$, where $K$ is the Gaussian
curvature at the point $P$ and $S(W)$ is the corresponding area.
The ratio $\frac{\omega(W)}{S(W)}$ is called the \textit{specific
curvature of the domain} $W$, and it is denoted by $\kappa(W)$.
 A convex surface has \textit{Gaussian curvature equal to $K$ at a point $P$} if the limit of the specific curvature of the domain tends to the limit $K$ whenever this domain shrinks to the point $P$.
\end{definition}

We recall the following fundamental results which have been
formulated and proved by A.D.~Alexandrov (see \cite[pp.~365, 377:
Theorems 2, 3, p.~7, 64--66: Theorem 3, p.~377: footnote
9]{Alexandrov:06}):

\vskip1mm (1) If the specific curvature $\kappa(W)$ of every
domain $W$ lying in some neighborhood of a point $P$ on a convex
surface does not exceed some positive number $K$ then there exists
$r_{0}>0$ such that it is possible to draw a shortest arc of
length at least $r_{0}$ from the point $P$ in each direction.

\vskip1mm (2) Suppose that a \textit{specific curvature} $\kappa$
of a convex surface $M$ is \textit{bounded} on a neighborhood $W$,
i.e., $K_{1}\le \kappa(W)\le K_{2}$. Introduce the polar geodesic
coordinates on the sphere $S_{K_{i}}\ (i=1,2)$, and construct a
mapping such that to each point of the neighborhood $W$ with
coordinates $r, \varphi,$ correspond points of the spheres
$S_{K_{i}}$ for $i=1,2,$ with the same coordinates. If, to a curve
$L$ on $M$, this mapping puts in correspondence the curves $L_{1}$
and $L_{2}$ on the spheres $S_{K_{1}}$ and $S_{K_{2}},$ then the
lengths of these curves are related by the inequalities
\[
 S(L_{1})\ge S(L)\ge S(L_{2}).
\]

(3) If the specific curvature $\kappa(W)$ is $\ge K$ ($\le K$) in
a triangle $\triangle ABC$, then the angles of $\triangle ABC$ are
no less (or greater) than the corresponding angles of the triangle
$\triangle\mathcal{ABC}$ on the $K$-plane (see for the definition
of $K$-plane in Section~\ref{sec:3}, p.~9). For $K>0,$ we consider
the perimeter of the triangle to be lower than
$\frac{2\pi}{\sqrt{K}}$.

\vskip1mm (4) Each two points of a complete convex surface can be
connected by a shortest arc.

\vskip1mm (5) If the specific curvature $\kappa(W)$ is $\le K$ on
a complete convex surface, then each arc of a geodesic of length
at most $\frac{\pi}{\sqrt{K}}$ is a shortest arc in general but
not compared with close lines.

\vskip1mm We consider $M$ to be a complete convex surface with
bounded specific curvature $K_{1}<\kappa(W)<K_{2}$. The length of
each shortest arc on $W$ is greater than
$\frac{\pi}{\sqrt{K_{2}}}$ and smaller than
$\frac{\pi}{\sqrt{K_{1}}}$.

\vskip1mm\hskip-2mm We state the weighted Fermat-Torricelli
(w.\,F-T) problem on $M$ for quadrilaterals.

\begin{problem} Let $ABCD\subset M$ be a convex quadrilateral whose
perimeter is less than $\frac{2\pi}{\sqrt{K_{1}}}$. Suppose that a
positive number (weight) $w_{R},$ corresponds to the vertex
$R\in\{A,B,C,D\}.$ Find the w.\,F-T point $P_{F}$ such that
\begin{equation}\label{minimumcs}
 f(P_{F})=w_{A}l_{A}+w_{B}l_{B}+w_{C}l_{C}+w_{D}l_{D} \to\ \min
\end{equation}
where $l_{R}$ is the length of the shortest arc from the w.\,F-T
point $P_{F}$ to the vertex $R\in\{A,B,C,D\}$
(Fig.~\ref{figureconvFTfirst}).
\end{problem}

\begin{figure}
\centering
\includegraphics[scale=0.30]{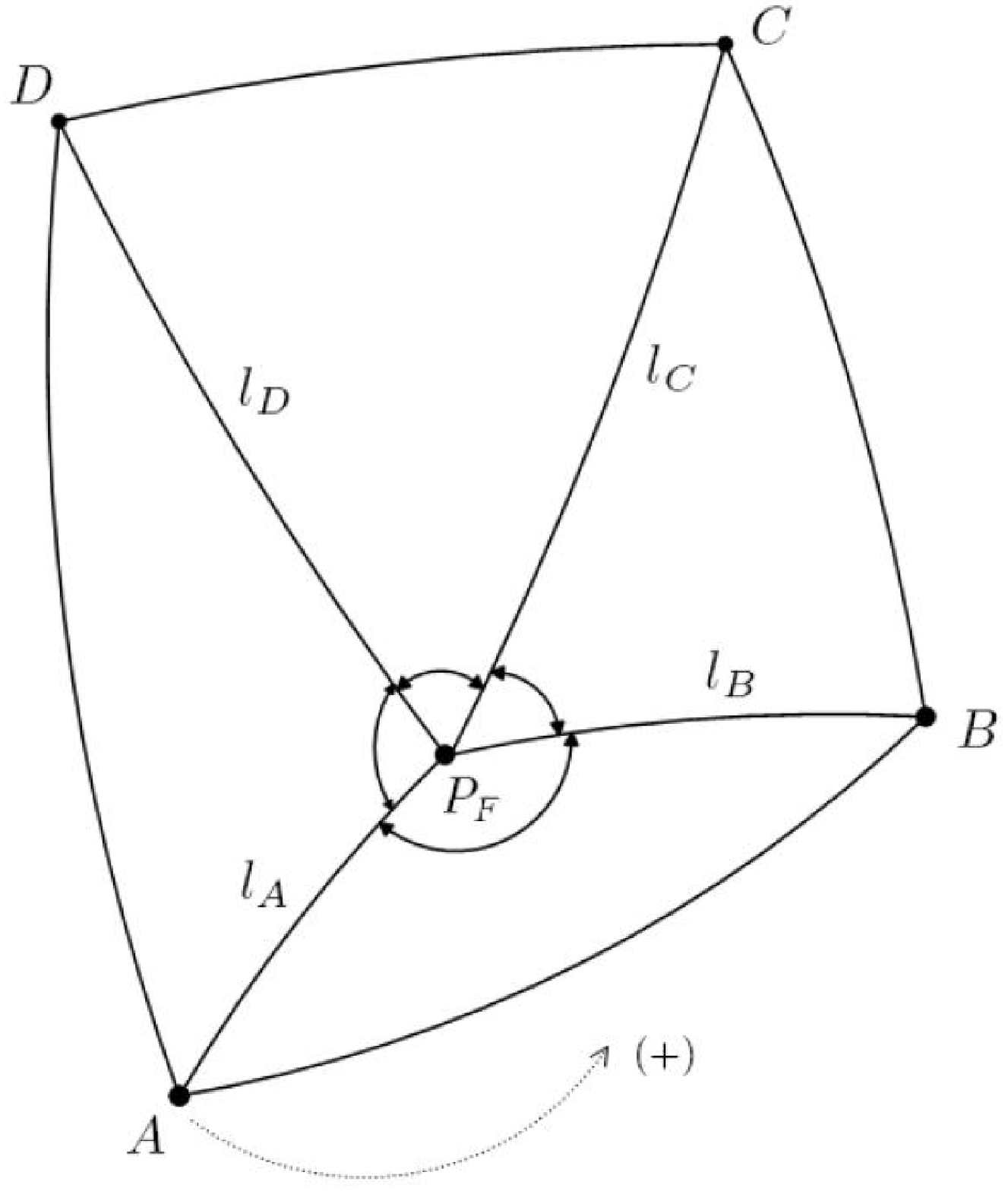}
\caption{} \label{figureconvFTfirst}
\end{figure}

In the paper we provide the plasticity equations for convex
quadrilaterals on a complete convex surface with bounded specific
curvature $M$ (Theorem~\ref{T-01} and Corollary~\ref{corol})  by
applying a method of differentiation for shortest arcs which has
been introduced  for the differentiation of the length of
geodesics in \cite{Cots/Zach:11}, \cite{Zachos/Cots:10} which
generalize the plasticity equations introduced in
\cite{Zachos/Zou:88} on the two dimensional K-plane (Two
dimensional sphere, hyperbolic plane and Euclidean plane) and we
prove the plasticity property of convex quadrilaterals on M which
was numerically verified for convex quadrilaterals on the two
dimensional K-plane in \cite{Zachos/Zou:88} without giving a
proof.

The main result of the paper is the plasticity principle of convex
quadrilaterals on $M$ (Theorem~\ref{plasticityprincipleM} in
Section~\ref{sec:2}) which states that:

Given four shortest arcs which meet at the weighted
Fermat-Torricelli point and their endpoints form a convex
quadrilateral and the weighted Fermat-Torricelli point belongs to
the interior of this convex quadrilateral, an increase of the
weight corresponding to a shortest arc causes a decrease of the
two weights that correspond to the two neighboring shortest arcs
and an increase of the weight corresponding to the opposite
shortest arc.

Using comparison geometry, we derive the plasticity equations for
convex quadrilaterals for two cases of manifolds with metric of
positive curvature which has been obtained by the gluing theorem
of A.D.~Alexandrov (Corollary~\ref{prpr1} in Section~\ref{sec:3}).
These results provide a generalization of the evolution of convex
quadrilaterals which has been introduced in \cite{Zachos/Zou:88}
on surfaces of constant Gaussian curvature.



\section{The 4-inverse weighted Fermat-Torricelli problem on a complete convex surface with bounded specific curvature}
\label{sec:2}

\begin{problem}
Given a point $P$ which belongs to the interior of $ABCD$ on $M$,
does there exist a unique set of positive weights $w_{i},$ such
that
\begin{displaymath}
 w_{A}+w_{B}+w_{C}+w_{D} = c =const,
\end{displaymath}
for which P minimizes
\begin{displaymath}
 f(P)=w_{A}l_{A}+w_{B}l_{B}+w_{C}l_{C}+w_{D}l_{D}
\end{displaymath}
where $l_{R}$ is the length of the shortest arc from $P$ to the
vertex $R\in\{A,B,C,D\}.$

 This is the 4-inverse w.\,F-T problem on $M$
(Fig.~\ref{figureconvFTfirst}).
\end{problem}
For $w_{D}=0,$ we derive the 3-inverse w.\,F-T problem which has
been introduced and solved in $\mathbb{R}^{2}$ by S. Gueron and R.
Tessler in \cite{Gue/Tes:02} and it has been solved by
\cite{Zach/Zou:08} on a surface of constant Gaussian Curvature and
generalized in \cite{Zachos/Cots:10}, \cite{Cots/Zach:11} on a
$C^{2}$ surface. The solution of the 3-inverse weighted
Fermat-Torricelli problem gives a positive answer on this problem
on a $C^{2}$ surface. It is called the 3-inverse normalized
weighted Fermat-Torricelli problem for $c=1$ (see
\cite{Gue/Tes:02}, page 449). We show that a solution of the
4-inverse weighted Fermat-Torricelli problem (
Corollary~\ref{corol}) gives a negative answer to this problem
because three weights depend on one variable weight (
Theorem~\ref{T-01}).

An application of the solution of the 4-inverse w.\,F-T problem is
the derivation of the evolution of convex quadrilaterals on $M$.
The evolution of convex quadrilaterals is given by the invariance
of the weighted Fermat-Torricelli point for given convex
quadrilaterals (geometric plasticity) and the plasticity of convex
quadrilaterals which are convex quadrilaterals which satisfy
Theorem~\ref{T-01} (dynamic plasticity). We note that the
geometric plasticity was known to Viviani for the case of n given
points in $\mathbb{R}^{2}.$
\begin{theorem}\label{T-01}
 Consider the 4-inverse w.\,F-T problem on a complete convex surface of bounded specific curvature $M$ in $\mathbb{R}^{3}.$
The plasticity of convex quadrilaterals on $M$ is given by the
following three equations:
\begin{eqnarray} \label{plastic1}
 (\frac{w_B}{w_A})_{ABCD}=(\frac{w_B}{w_A})_{ABC}[1-(\frac{w_D}{w_A})_{ABCD}(\frac{w_A}{w_D})_{ACD}],\\
\label{plastic2}
 (\frac{w_C}{w_A})_{ABCD}=(\frac{w_C}{w_A})_{ABC}[1-(\frac{w_D}{w_A})_{ABCD}(\frac{w_A}{w_D})_{ABD}],
\end{eqnarray}
and
\begin{equation}\label{invcond4}
 (w_{A})_{ABCD}+(w_{B})_{ABCD}+(w_{C})_{ABCD}+(w_{D})_{ABCD}=const.
\end{equation}
 The weight $(w_R)_{ABCD}$ corresponds to the vertex $R$ that
lie on the shortest arc $P_{F}R$, for $R\in\{A,B,C,D\}$ and the
weight $(w_S)_{SLN}$ corresponds to the vertex $S$ that lie on the
shortest arc $P_{F}S$ regarding the triangle $\triangle SLN$, for
$S,L,N\in\{A,B,C,D\}$ and $S\ne L\ne N.$
\end{theorem}

\begin{proof}[Proof of Theorem~\ref{T-01}:]
Firstly, we assume that we choose four initial (given) values
$(w_{R})_{ABCD}(0)$ concerning the weights $(w_{R})_{ABCD}$ for
$R\in\{A,B,C,D\}$ such that the w.\,F-T point $P_{F}$ exists and
it is located at the interior of $ABCD$.
Differentiation of the objective function with respect to a
specific arc length, yields the system
\[
 \sum_{Q\in\{A,B,C,D\}}w_{Q}\cos(\angle QP_{F}R)=0\qquad\mbox{for}\quad R\in\{A,B,C,D\}.
 \]
Without loss of generality,  we choose $Q\in\{A,B,C,D\}$. We
suppose that the point $Q$, the arc length parameterized curve
$c(t)$ and the point $P_{F}$ that lie in $c(t)$ are given in the
neighborhood $W_{P_{F}}.$


We apply the following general result given in \cite[Corollary
4.5.7, p.~125, Remark~4.5.8, p.~126 and Theorem 4.5.6,
p.~123]{Burago:01} that deals with a rule for differentiating the
length of a continuous family of shortest arcs connecting $Q$ to
points $c(t),$ for $Q\in\{A,B,C,D\}$ and the existence of the
right derivative $\frac{d}{dt}l_{Q}$, although the shortest arcs
from $Q$ to $c(t)$ may not be unique on a non-negatively curved
complete locally compact space $X$ ($M$ is a subset of $X$):

\vskip1mm Let $c(t):[0,T]\to M$ be arc length parameterized
geodesic, $Q$ a fixed point in $M$, $Q\ne c(0)$. Then the function
$t\to l_{Q}(t)=|Qc(t)|$ has the right derivative and
\begin{equation}\label{compute partialnn}
 \lim_{t\to+0}\frac{l_{Q}(t)-l_{Q}(0)}{t}
 =\cos(\pi-\varphi_{c,\min}^{\prime}(t))=\cos(\pi-\varphi_{c}^{\prime}(t)),
\end{equation}
where $\varphi_{c,\min}^{\prime}(t)$ is the infimum (minimum) of
angles between $c(t)$ and the shortest arcs connecting $c(0)$ to
$Q$ and takes a unique value $\varphi_{c}^{\prime}(t)$ on $M.$

We choose the parametrization (a similar technique has been used
in \cite{Cots/Zach:11}, \cite{Zachos/Cots:10} for the
parametrization of length of geodesics):
\begin{equation}\label{compute partialparametert}
 l_{A}(t)=t,
\end{equation}
We assume that the shortest arcs $l_{B}$, $l_{C},$ $l_{D},$  can
be expressed as functions of $l_{A}$:
\begin{equation}\label{eq:a_22}
 l_{B}=l_{B}(l_{A}),\quad l_{C}=l_{C}(l_{A}), \quad l_{D}=l_{D}(l_{A}).
\end{equation}
From (\ref{eq:a_22}) and (\ref{minimumcs}) the following equation
is obtained:
\begin{equation}\label{eq:B_2f_11}
 w_{A}l_{A}+w_{B}l_{B}(l_{A})+w_{C}l_{C}(l_{A})+w_{D}l_{D}(l_{A})= \min
\end{equation}
Taking into account that the point $P_{F}$ is the w.\,F-T point of
$ABCD$ we derive that the right derivative of (\ref{eq:B_2f_11})
with respect to the variable $l_{A}$ and considering (\ref{compute
partialparametert}), we get
\begin{equation} \label{eq:B221fract}
 w_{A}+w_{B}\lim_{t\to +0}
 \frac{l_{B}(t)-l_{B}(0)}{l_{A}(t)}+w_{C}\lim_{t\to +0}
 \frac{l_{C}(t)-l_{C}(0)}{l_{A}(t)}+w_{D}\lim_{t\to +0}
 \frac{l_{D}(t)-l_{D}(0)}{l_{A}(t)}=0.
\end{equation}
From (\ref{compute partialnn}) and (\ref{compute
partialparametert}), for $Q=B,$ we get
\begin{equation}\label{partialimp2t}
 \lim_{t\to +0}\frac{l_{B}(t)-l_{B}(0)}{l_{A}(t)}=\cos(\angle AP_{F}B).
\end{equation}
Let the point $C$, the length parameterized curve $c(t)$ and the
point $P_{F}$ that belongs in $c(t)$ be given in the neighborhood
$W_{P_{F}}$. Taking into account (\ref{compute partialnn}) and
(\ref{compute partialparametert}), for $Q=C,$ we have
(\cite[Corollary 4.5.7, p.~125, Remark 4.5.8, p.~126, Theorem
4.5.6, p.~123]{Burago:01}, \cite[Problem 1.5.3, p.~16 and Lemma
3.5.1 and Remark 3.5.1, pp.~164--165]{VToponogov:05})
\begin{equation}\label{partialimp3t}
 \lim_{t\to +0}\frac{l_{C}(t)-l_{C}(0)}{l_{A}}=\cos(\angle AP_{F}C).
\end{equation}
Similarly, let the point $D$, the length parameterized curve
$c(t)$ and the point $P_{F}$ that belongs in $c(t)$ be given in
the neighborhood $W_{P_{F}}$. Taking into account (\ref{compute
partialnn}) and (\ref{compute partialparametert}), for $Q=D$, we
have
\begin{equation}\label{partialimp4tb}
 \lim_{t\to +0}\frac{l_{D}(t)-l_{D}(0)}{l_{A}}=\cos(\angle AP_{F}D).
\end{equation}
Replacing (\ref{partialimp2t})\,--\,(\ref{partialimp4tb}) in
(\ref{eq:B221fract}), we obtain
\begin{equation}\label{equation1t}
 w_{A}+w_{B}\cos(\angle AP_{F}B)=-w_{C}\cos(\angle
 AP_{F}C)-w_{D}\cos(\angle AP_{F}D).
\end{equation}
Similarly, working cyclically and differentiating
(\ref{minimumcs}) with respect to $l_{B}$, and choosing the
parametrization $l_{B}(t^{\prime})=t^{\prime},$ by differentiating
(\ref{minimumcs}) with respect to $l_{C}$ and by choosing the
parametrization $l_{C}(t^{\prime\prime})=t^{\prime\prime}$ and by
differentiating (\ref{minimumcs}) with respect to $l_{D}$ and
choosing the parametrization
$l_{D}(t^{\prime\prime\prime})=t^{\prime\prime\prime}$, we derive
three relations, respectively,
\begin{eqnarray}\label{equation2t}
 w_{A}\cos(\angle BP_{F}A)+w_{B}=-w_{C}\cos(\angle BP_{F}C)-w_{D}\cos(\angle BP_{F}D),\\
\label{equation3t}
 w_{A}\cos(\angle CP_{F}A)+w_{B}\cos(\angle CP_{F}B)=-w_{C}-w_{D}\cos(\angle CP_{F}D),\\
\label{equation4tb}
 w_{A}\cos(\angle DP_{F}A)+w_{B}\cos(\angle DP_{F}B)=-w_{C}\cos(\angle DP_{F}C)-w_{D}.
\end{eqnarray}
The equations (\ref{equation1t})\,--\,(\ref{equation4tb}) could be
written in a unified form
\[
 \sum_{Q\in\{A,B,C,D\}}w_{Q}\cos(\angle QP_{F}R)=0\qquad\mbox{for}\quad R\in\{A,B,C,D\}.
\]
The points of $M$ are distinguished into three categories (see
\cite[p.~386]{Alexandrov:06}):

(1) conical points at which the tangent cone has complete angle
less than $2 \pi$,

(2) edge points at which the tangent cone is dihedral angle,

(3) ''smooth" points at which the tangent cone is a plane.

\vskip1.5mm
 The point $P_{F}$  which is the intersection of four prescribed shortest
 arcs in $M$ can be only a smooth point and it cannot be a conical point or an edge point, due to a result of A.D.~Alexandrov in \cite[Theorem 5, p.~135]{Alexandrov:06}.
 Therefore, the complete angle of $P_{F}$ is given by the following formula:
\begin{equation}\label{anglesumt}
 \angle AP_{F}B+\angle BP_{F}C+\angle CP_{F}D+\angle DP_{F}A= 2\pi.
\end{equation}
Taking into account the 4-inverse w.\,F-T problem, the weights
$(w_{R})_{ABCD}=w_{R}$ for $R\in\{A,B,C,D\}$ become variable
weights which satisfy (\ref{equation1t}) --
(\ref{equation4tb}).

 Taking into consideration the trigonometric identity
\[
 1-\cos(\angle AP_{F}B)\cos(\angle BP_{F}A)=-\sin(\angle AP_{F}B)\sin(\angle BP_{F}A)
\]
and the orientation of the angles with respect to $P_{F}$ taken
counterclockwise (see Fig.~\ref{figureconvFTfirst})  we solve the
linear system  (\ref{equation1t}) -- (\ref{equation2t}) with
respect to $w_{A}$ and $w_{B}$:
\begin{eqnarray}\label{ratio1t}
 w_{A}\sin(\angle BP_{F}A)+w_{C}\sin(\angle BP_{F}C)+w_{D}\sin(\angle BP_{F}D)=0,\\
 \label{ratio2t}
 w_{B}\sin(\angle AP_{F}B)+w_{C}\sin(\angle AP_{F}C)+w_{D}\sin(\angle AP_{F}D)=0.
\end{eqnarray}
 Similarly, solving the linear system (\ref{equation1t}) and (\ref{equation3t}) with respect to
 $w_{A},$ $w_{C},$ we obtain
\begin{equation}\label{ratio3t}
 w_{A}\sin(\angle CP_{F}A)+w_{B}\sin(\angle CP_{F}B)+w_{D}\sin(\angle CP_{F}D)=0.
\end{equation}
We write (\ref{ratio3t}) in the following form:
\[
 (\frac{w_B}{w_A})_{ABCD}=-\frac{\sin(\angle CP_{F}A)}{\sin(\angle CP_{F}B)} [(\frac{w_D}{w_A})_{ABCD}\frac{\sin(\angle CP_{F}D)}{\sin(\angle CP_{F}A)}+1].
\]
If we set $(w_{D})_{ABCD}=0,$ in (\ref{ratio1t}), (\ref{ratio2t})
and (\ref{ratio3t}), we obtain a solution of the 3-inverse w.\,F-T
problem for the triangle $\triangle ABC$ (see
\cite{Zachos/Cots:10}, for a solution of this problem on a
$C^{2}$-regular surface) and we get
\[
 -\frac{\sin(\angle CP_{F}A)}{\sin(\angle CP_{F}B)}=(\frac{w_B}{w_A})_{ABC}.
\]
Similarly, if we set $(w_{B})_{ABCD}=0,$ in
(\ref{ratio1t})\,--\,(\ref{ratio3t}), we obtain a solution of the
3-inverse w.\,F-T problem for the triangle $\triangle
ACD^{\star}$, where ($D^{\star}$ is the symmetric point of $D$
with respect to $P_{F}$, and we get
\[
 -(\frac{w_A}{w_D})_{ACD^{\star}}=\frac{\sin(\angle CP_{F}D^{\star})}{\sin(\angle CP_{F}A)}=-\frac{\sin(\angle CP_{F}D)}{\sin(\angle CP_{F}A)}=(\frac{w_A}{w_D})_{ACD},
\]
which gives (\ref{plastic1}).

We proceed by deriving (\ref{plastic2}).
We write (\ref{ratio1t}) in the form
\[
 (\frac{w_C}{w_A})_{ABCD}=-\frac{\sin(\angle BP_{F}A)}{\sin(\angle BP_{F}C)} (1+(\frac{w_D}{w_A})_{ABCD}\frac{\sin(\angle BP_{F}D)}{\sin(\angle BP_{F}A)}).
\]
If we set $(w_{D})_{ABCD}=0,$ in (\ref{ratio1t}) --
(\ref{ratio3t}), we obtain a solution of the 3-inverse w.\,F-T
problem for the triangle $\triangle ABC$ and we get
\[
 (\frac{w_C}{w_A})_{ABC}=-\frac{\sin(\angle BP_{F}A)}{\sin(\angle BP_{F}C)}.
\]
Similarly, if we set $(w_{C})_{ABCD}=0,$ in
(\ref{ratio1t})\,--\,(\ref{ratio3t}), we obtain a solution of the
3-inverse w.\,F-T problem for the triangle $\triangle ABD$), and
we get
\[
 (\frac{w_A}{w_D})_{ABD}=-\frac{\sin(\angle BP_{F}D)}{\sin(\angle BP_{F}A)},
\]
which yields (\ref{plastic2}).
\end{proof}

As a direct consequence of Theorem~\ref{T-01} we obtain

\begin{corollary}\label{corol} If $\,\sum_{ABCD}w=\sum_{ABC}w=\sum_{ABD}w=\sum_{ACD}w=\sum_{BCD}w$, where
$$\sum\nolimits_{ABCD}w:=(w_A)_{ABCD}(1+\frac{w_B}{w_A}+\frac{w_C}{w_A}+\frac{w_D}{w_A})_{ABCD},$$
then
\begin{equation}\label{E-wabcd}
 (w_i)_{ABCD}=a_i (w_D)_{ABCD}+ b_i,\quad i\in\{A,B,C\},
\end{equation}
where
\begin{eqnarray}\label{E-ab}
\nonumber
 (a_A,\,b_A)\eq  (\frac{(\frac{w_A}{w_D})_{ACD}(\frac{w_B}{w_A})_{ABC}
 +(\frac{w_A}{w_D})_{ABD}(\frac{w_C}{w_A})_{ABC}-1}
 {1+(\frac{w_B}{w_A})_{ABC}+(\frac{w_C}{w_A})_{ABC}},\ (w_A)_{ABC}),\\
\nonumber
 (a_B,\,b_B)\eq  (a_{A}(\frac{w_B}{w_A})_{ABC}-(\frac{w_A}{w_D})_{ACD}(\frac{w_B}{w_A})_{ABC},\ (w_B)_{ABC}),\\
 (a_C,\,b_C)\eq  (a_{A}(\frac{w_C}{w_A})_{ABC}-(\frac{w_A}{w_D})_{ABD}(\frac{w_C}{w_A})_{ABC},\ (w_C)_{ABC}).
\end{eqnarray}
\end{corollary}


We continue by proving the main result, which we call \textit{a
plasticity principle of convex quadrilaterals on $M$}:

\begin{theorem}\label{plasticityprincipleM}
Given four shortest arcs which meet at the weighted
Fermat-Torricelli point $P_{F}$ and their endpoints form a convex
quadrilateral on $M$ and the weighted Fermat-Torricelli point
belongs to the interior of this convex quadrilateral, an increase
of the weight that corresponds to a shortest arc causes a decrease
to the two weights that correspond to the two neighboring shortest
arcs and an increase to the weight that corresponds to the
opposite shortest arc.
\end{theorem}

\begin{proof}[Proof of Theorem~\ref{plasticityprincipleM}:]
We take into account the three plasticity equations of
Theorem~\ref{T-01} and Corollary~\ref{corol} which are derived by
applying the conditions $\sum_{ABCD}
w:=(w_A)_{ABCD}(1+\frac{w_B}{w_A}+\frac{w_C}{w_A}+\frac{w_D}{w_A})_{ABCD}$
and $\sum_{ABCD} w=\sum_{ABC} w=\sum_{ABD} w=\sum_{ACD}
w=\sum_{BCD} w = const$ such that
(\ref{E-wabcd})\,--\,(\ref{E-ab}) hold. We prove that $a_{A},
a_{C}<0$ and $a_{B}>0.$ We calculate the coefficient $a_{A}$:
\begin{eqnarray*}
 a_{A}\eq \frac{(\frac{w_A}{w_D})_{ACD}(\frac{w_B}{w_A})_{ABC}+(\frac{w_A}{w_D})_{ABD}(\frac{w_C}{w_A})_{ABC}-1
 }{1+(\frac{w_B}{w_A})_{ABC}+(\frac{w_C}{w_A})_{ABC}}\\
 \eq \frac{\frac{\sin(\angle CP_{F}D)}{\sin(\angle CP_{F}A)}\frac{\sin(\angle CP_{F}A)}{\sin(\angle CP_{F}B)}+\frac{\sin(\angle BP_{F}D)}{\sin(\angle
 BP_{F}A)}\frac{\sin(\angle BP_{F}A)}{\sin(\angle BP_{F}C)} -1}{\frac{const}{(w_{A})_{ABC}}}\\
 \eq \frac{\frac{\sin(\angle C^{\star}P_{F}D)}{\sin(\angle C^{\star}P_{F}B)}
 -\frac{\sin(\angle BP_{F}D)}{\sin(\angle BP_{F}C^{\star})} -1}{\frac{const}{(w_{A})_{ABC}}}
 =\frac{ -(\frac{w_B}{w_D})_{BC^{\star}D}+(\frac{w_{C^{\star}}}{w_D})_{BC^{\star}D}-1       }  {\frac{const}{(w_{A})_{ABC}}}\\
 \eq \frac{ -(\frac{w_B}{w_D})_{BCD}-(\frac{w_C}{w_D})_{BCD}-1}{\frac{const}{(w_{A})_{ABC}}}
 =-\frac{(w_{A})_{ABC}}{(w_{D})_{BCD}}<0,
\end{eqnarray*}
 because $(w_{A})_{ABC}$ and $(w_{D})_{BCD}$ are positive numbers.
The point $C^{\star}$ is the symmetric point of $C$ with respect
to $P_{F}$ such that $l_{P_{F}}(C)=l_{P_{F}}(C^{\star})$. Taking
into account that $P_{F}$ is located at the interior of the
triangle $\triangle BC^{\star}D,$ we have
\[
 (w_{B})_{BC^{\star}D}+(w_{C^{\star}})_{BC^{\star}D}+(w_{D})_{BC^{\star}D}=c,
\]
where $(w_{R})_{BC^{\star}D}$ are positive numbers for
$R\in\{B,C^{\star},D\}$.

Concerning the triangle $\triangle BCD,$ $P_{F}$ is not located at
the interior of $P_{F}$ and by the relation
\[
 (w_{B})_{BCD}+(w_{C})_{BCD}+(w_{D})_{BCD}=c,
\]
we obtain
\[
 (w_{C})_{BCD}=-(w_{C^{\star}})_{BC^{\star}D}<0,\qquad(w_{B})_{BCD}, (w_{D})_{BCD}>0.
\]
Therefore,
\[
 a_{C}=a_A(\frac{w_C}{w_A})_{ABC}-(\frac{w_A}{w_D})_{ABD}(\frac{w_C}{w_A})_{ABC}<0.
\]
because $P_{F}$ is located at the interior of the triangles
$\triangle ABC,$ $\triangle ABD,$ which makes the weights
$(w_{A})_{ABC},$ $(w_{C})_{ABC},$ $(w_{A})_{ABD}$, $(w_{D})_{ABD}$
positive numbers.

We shall show $a_{B}>0$. From the 4-inverse w.\,F-T condition of
Problem 2 we~have
\begin{equation}\label{invwa}
(w_{A})_{ABCD}=c-(w_{B})_{ABCD}-(w_{C})_{ABCD}-(w_{D})_{ABCD}.
\end{equation}
Replacing (\ref{invwa}) in (\ref{ratio3t}), we get
\begin{eqnarray}\label{equation2wbwc}
 (w_{B})_{ABCD}(\sin(\angle CP_{F}B)-\sin(\angle CP_{F}A))
 +2(w_{C})_{ABCD}\sin(\angle AP_{F}C)=\nonumber\\
 w_{D}(\sin(\angle CP_{F}A)-\sin(\angle CP_{F}D))-c \sin(\angle CP_{F}A).
\end{eqnarray}
From (\ref{ratio2t}) and (\ref{equation2wbwc}) we derive that
\begin{eqnarray}\label{wb}
\nonumber
 (w_{B})_{ABCD}\eq \frac{\sin(\angle CP_{F}A)-\sin(\angle
 CP_{F}D)+2\sin(\angle AP_{F}D)}{\sin(\angle CP_{F}B)-\sin(\angle
 CP_{F}A)-2\sin(\angle AP_{F}B)} (w_{D})_{ABCD}\\
 &-&\frac{c \sin(\angle CP_{F}A)}{\sin(\angle CP_{F}B)-\sin(\angle CP_{F}A)-2\sin(\angle AP_{F}B)}.
\end{eqnarray}
Taking into account the counterclockwise orientation of angles
with respect to $P_{F}$ (see Fig.~\ref{figureconvFTfirst}), we get
\[
 \sin(\angle CP_{F}B)-\sin(\angle CP_{F}A)-2\sin(\angle AP_{F}B)<0,
\]
because
\[
 \sin(\angle CP_{F}B)<0,\quad
 -\sin(\angle CP_{F}A)<0,\quad
 -\sin(\angle AP_{F}B)<0.
\]
Similarly, taking into account the counterclockwise orientation of
angles with respect to $P_{F}$ (see Fig.~\ref{figureconvFTfirst}),
we get
\begin{eqnarray*}
 &&\hskip-5mm\sin(\angle CP_{F}A)-\sin(\angle CP_{F}D)+2\sin(\angle AP_{F}D)\\
 &&=-\sin(\angle AP_{F}C)-\sin(\angle CP_{F}D)+\sin(\angle AP_{F}C+\angle CP_{F}D)
 +\sin(\angle AP_{F}D)\\
 &&=\sin(\angle AP_{F}C)(\cos(\angle CP_{F}D)-1)
 +\sin(\angle CP_{F}D)(\cos(\angle AP_{F}C)-1)+\sin(\angle AP_{F}D)\\
 &&\le  \sin(\angle AP_{F}D)<0.
\end{eqnarray*}
Therefore, we have
\[
 a_{B}=\frac{\sin(\angle CP_{F}A)-\sin(\angle
 CP_{F}D)+2\sin(\angle AP_{F}D)}{\sin(\angle CP_{F}B)-\sin(\angle
 CP_{F}A)-2\sin(\angle AP_{F}B)}>0
\]
that completes the proof.
\end{proof}

 A numerical verification of the plasticity principle is given in \cite[Examples 4.7, 4.10, pp.~418--419]{Zachos/Zou:88}.

\begin{remark}\rm
Concerning the plasticity of convex quadrilaterals on $M$, we have
assumed that $P_{F}$ is located at the interior of $\triangle
ABC$. Then a fourth shortest arc "grows" from $P_{F}$ such that
its length reaches $l_{P_{F}}(D)$ and $ABCD$ is a convex
quadrilateral on $M$. Therefore, we consider that the 4-inverse
w.\,F-T problem is derived by the 3-inverse w.\,F-T problem on
$M$.
\end{remark}

We mention two particular cases (Proposition~\ref{spec1},
Example~\ref{spec2}) such that $P_{F}$ belongs to at least one of
the diagonals (shortest arc) of $ABCD$ where the 4-inverse w.\,F-T
problem is not derived by the 3-inverse w.\,F-T problem like in
Theorem~\ref{plasticityprincipleM}.

\begin{proposition}\label{spec1}
If $P_{F}$ belongs to the shortest arc $(B,D)$ and does not
necessarily belong to the shortest arc $(A,C)$, then the
plasticity principle of the convex quadrilateral $ABCD$ on $M$
holds.
\end{proposition}

\begin{proof}[Proof of Proposition~\ref{spec1}:]
Assuming that $P_{F}\in (B,D)\setminus([B,D]\cap[A,C])$, the angle
$\angle BP_FD$ is $\pi$ and from (\ref{ratio1t}), we have
\begin{equation}\label{waa}
 w_{A}= - \frac{\sin(\angle BP_{F}C)}{\sin(\angle BP_{F}A)} w_{C}.
\end{equation}
Replacing (\ref{waa}) to the inverse condition
\[
 w_{A}+w_{B}+w_{C}+w_{D}=c,
\]
we get
\begin{equation}\label{2.17bis}
 \big(1 -\frac{\sin(\angle BP_{F}C)}{\sin(\angle BP_{F}A)}\big) w_{C} +w_{B}=c-w_{D}.
\end{equation}
Replacing (\ref{waa}) in (\ref{ratio3t}), we have
\begin{equation} \label{2.19bis}
 -\frac{\sin(\angle CP_FA)\sin(\angle BP_{F}C)}{\sin(\angle BP_{F}A)} w_{C}+w_{B}\sin(\angle CP_FB)
 =-\sin(\angle CP_FD)w_{D}.
\end{equation}
Solving (\ref{2.17bis}) and (\ref{2.19bis}) with respect to
$w_{C}$ and $w_{A}$, we derive that
\[
 w_{C}=x_{C} w_{D}+ y_{C},\qquad
 w_{B}=x_{B} w_{D}+y_{B},
\]
where $x_{C},y_{C},x_{B},y_{B}$ are constant numbers
such that:
\begin{eqnarray}\label{xC}
 x_{C}\eq\frac{-\sin(\angle CP_{F}B)+\sin(\angle CP_{F}D)}{Det},\\
\label{xB}
 x_{B}\eq\frac{-\sin(\angle CP_{F}D)
 \big(1-\frac{\sin(\angle BP_{F}C)}{\sin(\angle BP_{F}A)}\big)
 +\sin(\angle CP_FA)\frac{\sin(\angle BP_{F}C)}{\sin(\angle BP_{F}A)}}{\rm Det},
\end{eqnarray}
 and Det is the determinant of (\ref{2.17bis}) and (\ref{2.19bis}):
\[
 {\rm Det}= \big(1 -\frac{\sin(\angle BP_{F}C)}{\sin(\angle BP_{F}A)}\big)\sin(\angle CP_FB)
 +\sin(\angle CP_FA) \frac{\sin(\angle BP_{F}C)} {\sin(\angle BP_{F}A)}.
\]
We will prove that $x_{C}<0$ and $x_{B}>0$.

The determinant Det of (\ref{2.17bis}) and (\ref{2.19bis}) is a
negative because
 $\sin(\angle CP_{F}B)<0,\
 \frac{\sin(\angle BP_{F}C)}{\sin(\angle BP_{F}A)}<1,\
 \sin(\angle BP_{F}A)<0,\
 \sin(\angle CP_FA)>0,\
 \sin(\angle BP_{F}C)>0$.
Taking into account that the numerator of (\ref{xC})\[-\sin(\angle
CP_{F}B)+\sin(\angle CP_{F}D)>0\] because
\[
 -\sin(\angle CP_{F}B)>0,\quad \sin(\angle CP_{F}D)>0,
\]
we obtain that $x_{C}<0.$
Taking into account that Det$<0$ and that the numerator of
(\ref{xB})
\[
 -\sin(\angle CP_{F}D)\big(1-\frac{\sin(\angle BP_{F}C)}{\sin(\angle BP_{F}A)}\big)
 +\sin(\angle CP_FA)\frac{\sin(\angle BP_{F}C)}{\sin(\angle BP_{F}A)}<0
\]
because
\[
 -\sin(\angle CP_{F}D)(1-\frac{\sin(\angle BP_{F}C)}{\sin(\angle BP_{F}A)})<0,
\]
and that
\[
 \sin(\angle CP_FA)\frac{\sin(\angle BP_{F}C)}{\sin(\angle BP_{F}A)}<0
\]
because $\sin(\angle BP_FA)<0$ (If $P^{\prime}$ is the point where
the diagonals meet we take $P_F\in[B,D]$ to be located closer to
$B$ with respect to $P^{\prime}$), we obtain that $x_{B}>0$.

The following equations
\[
 w_{C}=x_{C} w_{D}+ y_{C},\quad
 w_{B}=x_{B}w_{D}+y_{B},
 \]
 where $x_{C}<0$ and $x_{B}>0,$
give the plasticity principle of $ABCD$, because by increasing
$w_{D}$, the weight $w_{C}$ will be decreased and the weight
$w_{B}$ will be increased and considering the relation (\ref{waa})
$w_{A}$ will also be decreased because $-\frac{\sin(\angle
BP_{F}C)}{\sin(\angle BP_{F}A)}$ is a positive number.
\end{proof}

\begin{example}\label{spec2}\rm
If $P_{F}$ is the intersection of the two diagonals $(B,D)$ and
$(A,C)$, then the plasticity principle of $ABCD$ on $M$ holds.

To show this directly, assume that $P_{F}\in [B,D]\cap[A,C]$, and
get
\[
 \angle BP_{F}D=\angle CP_{F}A=\pi.
\]
From (\ref{ratio1t}), (\ref{ratio3t}) we derive $w_{A}=w_{C}$ and
$w_{B}=w_{D},$ respectively.

Replacing these two relations in  the 4-inverse condition of
Problem 2, we have
\[
 2 w_{A}+2 w_{D}=c,
\]
where $c$ is a positive real number.
 By the last relation, an increase in $w_{D}=w_{B}$ will cause a decrease in $w_{A}=w_{C},$ otherwise the inverse condition will not hold.
\end{example}

\begin{remark}\rm

 We would like to mention a
reformulation of the 4 inverse w.\,F-T problem for a convex
quadrilateral $ABCD$ on $\mathbb{R}^2,$ which is communicated to
the author by Professor Dr. Vladimir Rovenski.

Let $\Pi =\{ w_{A}+w_{B}+w_{C}+w_{D} =1\}$ (say, $c=1$) be a
3-plane in the linear space $\mathbb{R}^4$ of variables
$(w_{A},w_{B},w_{C},w_{D})$. Denote $v_{1}=PA/|PA|,\
v_{2}=PB/|PB|,\ v_{3}=PC/|PC|,\ v_{4}=PD/|PD|$ -- the unit vectors
in $\mathbb{R}^2$. The linear system $w_{A}v_{1}+w_{B}v_{2}+
w_{C}v_{3}+ w_{D}v_{4}=0$ (of two equations with respect to
$(x,y)$ given coordinates) determines a straight line
$w_{R}(t)=w_{R}(0) +t\vec{e}$ in $\Pi$, where
$e=(e_1,e_2,e_3,e_4)$ is a direction vector of the line.

Taking the inner product of $w_{A}v_{1}+w_{B}v_{2}+ w_{C}v_{3}+
w_{D}v_{4}=0$ with respect to the vector $v_{i}$ for $i=1,2,3,4,$
we obtain (\ref{equation1t})\,--\,(\ref{equation4tb}).

Taking the exterior product of
$w_{A}v_{1}+w_{B}v_{2}+w_{C}v_{3}+w_{D}v_{4}=0$ with respect to
the vector $v_{i}$ for $i=1,2,3,4,$ we obtain
(\ref{ratio1t})\,--\,(\ref{ratio3t}).
\end{remark}

\section{The plasticity of quadrilaterals on a manifold of positive curvature}
\label{sec:3}

We give some evolutionary structures for quadrilaterals on a
manifold with positive curvature. A metric of positive curvature
is intrinsic and the sum of the lower angles of every sufficiently
small convex triangle cannot be less than $\pi$ (see
\cite[p.~279]{Alexandrov:06}). The evolutionary structure for
quadrilaterals is derived by the plasticity equations and the
4-inverse Fermat-Torricelli problem for the following two cases:

\vskip1.5mm (1) A manifold $M^{\prime}$ which is obtained by
gluing two triangles that exist on a two-dimensional sphere with
constant Gaussian curvature $K_{1}$ and two triangles that exist
on a two-dimensional sphere of constant Gaussian curvature
$K_{2},$ for $K_{1}<K_{2},$ and the sum of angles meeting at the
Fermat-Torricelli point $P_{F}$ of the quadrilateral equals
$2\pi,$

\vskip1.5mm (2) A manifold $M^{\prime \prime}$ which is obtained
by gluing two triangles that exist on a complete convex surface of
bounded specific curvature $M\ (K_{1}<\kappa<K_{2})$ and two
triangles that exist on a two-dimensional sphere of constant
Gaussian curvature $K_{1},$ and $K_{2},$ respectively, and the sum
of angles meeting at the Fermat-Torricelli point $P_{F}$ of the
quadrilateral equals $2\pi.$

\vskip1.5mm These two cases form a manifold with metric of
positive curvature because they are specific cases of the gluing
theorem of A.D.~Alexandrov (see
\cite[pp.~278--279]{Alexandrov:06}) which states that:

\vskip1.5mm \textit{If a manifold is obtained by gluing a finite
number of polygons with metric of positive curvature so that, at
each vertex, the sum of angles of these polygons meeting at it is
no greater than $2\pi,$ then the metric on the whole manifold is
also a metric of positive curvature}.

A further generalization of the gluing theorem of A.D.~Alexandrov
was made by Yu. Reshetnyak concerning the gluing of
Cartan-Alexandrov-Toponogov spaces with curvature (in the sense of
A.D.~Alexandrov) bounded from above by a real number $K_{2}$
(CAT$(K_{2})$) spaces along proper convex subsets is a
CAT$(K_{2})$ space (see \cite[pp.~188--189 and footnote
15]{VToponogov:05} and \cite{BNik:07}).

We mention the definition of a $K_{0}$-plane and the comparison
triangle of a triangle on a two-dimensional surface.

If $K_{0}$ denotes the constant Gaussian curvature of a surface
$M$, then $M$ is called the $K_{0}$-plane.

\vskip1.5mm If $K_{0}<0,$ the $K_{0}$-plane is a Lobachevski
(hyperbolic) plane $H^{2}$.

\vskip1.5mm If $K_{0}=0,$ the $K_{0}$-plane is an Euclidean plane
$\mathbb{R}^{2}$.

\vskip1.5mm If $K_{0}>0,$ the $K_{0}$-plane is an open hemisphere
$S^{2}$ of radius $\frac{1}{\sqrt{K_{0}}}$.

\vskip1.5mm A \textit{comparison triangle} on a $K_{0}$-plane
denoted by $(\triangle \mathcal{ABC})_{K_{0}}$ of  $\triangle ABC$
on a two-dimensional surface is a triangle whose corresponding
sides have equal lengths: $l_{\mathcal{A}}(\mathcal{B})=l_{A}(B),$
$l_{\mathcal{B}}(\mathcal{C})=l_{B}(C),$
$l_{\mathcal{A}}(\mathcal{C})=l_{A}(C)$ (\cite[pp.~185,
188]{VToponogov:05}). The existence of the comparison triangle
$\triangle(\mathcal{ABC})_{K_{0}},$ for $K_{0}>0,$ is given by the
condition that the perimeter of the triangle is not greater than
$\frac{2\pi}{\sqrt{K_{0}}}.$

The following corollary is given on $M^{\prime}$ and
$M^{\prime\prime}$ which are particular cases of M considered in
Theorem~\ref{T-01}.

\begin{corollary}\label{prpr1}
Consider the 4-inverse w.\,F-T problem on $M^{q}$ for
$q\in\{\prime,\prime\prime.\}$ The following equations point out
the plasticity of quadrilaterals $A^{q}B^{q}C^{q}D^{q}$ and on
$M^{q}$:
\begin{equation} \label{plastic1M1}
(\frac{w_{B^{q}}}{w_{A^{q}}})_{A^{q}B^{q}C^{q}D^{q}}=(\frac{w_{B^{q}}}{w_{A^{q}}})_{A^{q}B^{q}C^{q}}[1-(\frac{w_{D^{q}}}{w_{A^{q}}})_{A^{q}B^{q}C^{q}D^{q}}
(\frac{w_{A^{q}}}{w_{D^{q}}})_{A^{q}C^{q}D^{q}}],
\end{equation}
\begin{equation} \label{plastic2M1}
(\frac{w_{C^{q}}}{w_{A^{q}}})_{A^{q}B^{q}C^{q}D^{q}}=(\frac{w_{C^{q}}}{w_{A^{q}}})_{A^{q}B^{q}C^{q}}[1-(\frac{w_{D^{q}}}{w_{A^{q}}})_{A^{q}B^{q}C^{q}D^{q}}
(\frac{w_{A^{q}}}{w_{D^{q}}})_{A^{q}B^{q}D^{q}}],
\end{equation}
and
\begin{equation}\label{invcond4M1}
(w_{A^{q}})_{A^{q}B^{q}C^{q}D^{q}}+(w_{B^{q}})_{A^{q}B^{q}C^{q}D^{q}}+(w_{C^{q}})_{A^{q}B^{q}C^{q}D^{q}}+(w_{D^{q}})_{A^{q}B^{q}C^{q}D^{q}}=const
\end{equation}
where the weight $(w_R)_{A^{q}B^{q}C^{q}D^{q}}$ corresponds to the
vertex $R$ that lie on the shortest arc $P_{F}R$,
$R\in\{A^{q},B^{q},C^{q},D^{q}\}$ and the weight $(w_S)_{SLN}$
corresponds to the vertex $S$ that lie on the shortest arc
$P_{F}S$ regarding the triangle $\triangle SLN$, for
$S,L,N\in\{A^{q},B^{q},C^{q},D^{q}\}$ and $S\ne L\ne N,$ for
$q\in\{\prime,\prime\prime\}.$
\end{corollary}

\begin{proof}[Proof of Corollary~\ref{prpr1} (Case $M^{\prime}$):]
Gluing the comparison triangles
 $(\triangle A^{\prime}P_{F}D^{\prime})_{K_{1}}$,
 $(\triangle D^{\prime}P_{F}C^{\prime})_{K_{1}}$,
 $(\triangle C^{\prime}P_{F}B^{\prime})_{K_{2}}$,
 $(\triangle B^{\prime}P_{F}A^{\prime})_{K_{2}}$ of
 $\triangle AP_{F}D$, $\triangle DP_{F}C$, $\triangle CP_{F}B$, and $\triangle BP_{F}A$,
respectively, we obtain the following angular relations on
$M^{\prime}:$
\begin{eqnarray*}
 &&\angle A^{\prime}P_{F}D^{\prime}=\angle AP_{F}D-\epsilon_{1},\quad
 \angle D^{\prime}P_{F}C^{\prime}=\angle DP_{F}C-\epsilon_{2},\\
 &&\angle C^{\prime}P_{F}B^{\prime}=\angle CP_{F}B+\epsilon_{3},\quad
 \angle B^{\prime}P_{F}A^{\prime}=\angle BP_{F}A+\epsilon_{4},
\end{eqnarray*}
(where $\epsilon_{i},$  are non-negative real numbers for
$i=1,2,3,4$) such that
\[
 \angle A^{\prime}P_{F}D^{\prime}+\angle D^{\prime}P_{F}C^{\prime}+\angle C^{\prime}P_{F}B^{\prime}+\angle B^{\prime}P_{F}A^{\prime}=2\pi.
\]
The angular relations hold due to a result of A.D.~Alexandrov also
known as the angle comparison theorem which states that:

\vskip1.5mm Let $T$ be a triangle on a complete convex surface of
bounded specific curvature and let triangle $T_{1},$ $T_{2}$ be
triangles on the two-dimensional spheres $S_{K_{1}}$ and
$S_{K_{2}}$ with sides of the same length as $T$. If $\alpha,$
$\alpha_{1},$ $\alpha_{2}$ are corresponding angles of these
triangles then $\alpha_{1}\le \alpha\le \alpha_{2}$ (see
\cite[Theorem 4, p.~377, p.~54, case 1]{Alexandrov:06}).

We apply the following general result which was given in
\cite[Corollary 4.5.7, p.~125, Remark 4.5.8, p.~126 and Theorem
4.5.6, p.~123]{Burago:01} that deals with a rule for
differentiating the length of a continuous family of shortest arcs
connecting $Q$ to points $c(t),$ for $Q\in\{A,B,C\}$ and the
existence of the right derivative $\frac{dl_{Q}}{dt},$  on a
non-negatively curved complete locally compact space $X$
($M^{\prime}$ a subset of $X$):

\vskip1.5mm Let $c(t):[0,T]\to M$ be a geodesic parameterized by
arclength, $Q$ a fixed point in $M$, $Q\ne c(0)$. Then the
function $t\to l_{Q}(t)=|Qc(t)|$ has the right derivative and
\[
 \lim_{t\to +0}\frac{l_{Q}(t)-l_{Q}(0)}{t}=\cos(\pi-\varphi_{c,\min}^{\prime}(t)),
\]
where $\varphi_{c,\min}^{\prime}(t)$ is the infimum (minimum) of
angles between $c(t)$ and the shortest arcs connecting $c(0)$ to
$Q$. Taking into considerations the parameterization $l_{Q}(t)=t,$
for $Q\in\{A^{\prime},B^{\prime},C^{\prime},D^{\prime}\}$ and
differentiating the objective function (\ref{minimumcs}) with
respect to $t$ and following the same process in the proof of
Theorem~\ref{T-01}, we get
\[
 \sum_{Q\in\{A^{\prime},B^{\prime},C^{\prime},D^{\prime}\}}w_{Q}\cos(\angle
 QP_{F}R)=0\quad
 \mbox{for}\quad
 Q, R\subset\{A^{\prime},B^{\prime},C^{\prime},D^{\prime}\}\ \mbox{and}\  W_{A}\subset M.
\]
Applying the same steps that have been used in the proof of
Theorem~\ref{T-01}, we obtain (\ref{plastic1M1}) and
(\ref{plastic2M1}). This system of equations induces a
\textit{comparative plasticity} between the plasticity of $ABCD$
on $M$ and $A^{\prime}B^{\prime}C^{\prime}D^{\prime}$ on
$M^{\prime}$ because we may compare the plasticity equations on a
convex surface M with the plasticity equations on a convex surface
$M^{\prime}$ such that the given directions with respect to
$P_{F}$ vary in a specific way that depend on $\angle
R^{\prime}P_{F}S^{\prime}$ for
$R^{\prime},S^{\prime}\in\{A^{\prime},B^{\prime},C^{\prime},D^{\prime}\}:$
\begin{eqnarray*}
 &&\angle A^{\prime}P_{F}D^{\prime}=\angle AP_{F}D-\epsilon_{1},\quad
 \angle D^{\prime}P_{F}C^{\prime}=\angle DP_{F}C-\epsilon_{2},\\
 &&\angle C^{\prime}P_{F}B^{\prime}=\angle CP_{F}B+\epsilon_{3},\quad
 \angle B^{\prime}P_{F}A^{\prime}=\angle BP_{F}A+\epsilon_{4},
\end{eqnarray*}
and every ratio
$\frac{(\frac{w_{R^{\prime}}}{w_{S^{\prime}}})_{A^{\prime}B^{\prime}C^{\prime}D^{\prime}}}{(\frac{w_{R}}{w_{S}})_{ABCD}}$
 depend on the values of $\epsilon_{i}$ for $i=1,2,3,4$ and $R,S\in\{A,B,C,D\}.$

For instance by taking into account that
\[a_{A}+a_{B}+a_{C}=-1,\] and let $a_{A}$, $a_{C}$ be increased by $\delta a_{A},\delta
a_{C}>0$ then $a_{B}$ will be decreased by $\delta a_{B}<0$ such
that
\[a_{A}+\delta a_{A}+a_{B}+\delta a_{B}+a_{C}+\delta a_{C}=-1,\]
or
\[\delta a_{A}+\delta a_{B}+\delta a_{C}=0,\]
where $a_{R^{\prime}}=a_{R}+\delta a_{R},$  where $a_{R}$ is taken
by Corollary~\ref{corol}, for $R\in\{A,B,C\}.$
\end{proof}

\begin{proof}[Proof of Corollary~\ref{prpr1} (Case $M^{\prime\prime}$):]
Taking into consideration the gluing of the triangles $\triangle
AP_{F}D,$ $\triangle DP_{F}C$ which exist on $M$ with the two
comparison triangles $(\triangle
C^{\prime}P_{F}B^{\prime})_{K_{1}},$ $(\triangle
B^{\prime}P_{F}A^{\prime})_{K_{2}}$ of $\triangle CP_{F}B,$
$\triangle BP_{F}A$, respectively, we obtain the following angular
relations on $M^{\prime\prime}:$
\begin{eqnarray*}
 &&\angle A^{\prime\prime}P_{F}D^{\prime\prime}=\angle AP_{F}D,\quad
 \angle D^{\prime\prime}P_{F}C^{\prime\prime}=\angle DP_{F}C,\\
 &&\angle C^{\prime\prime}P_{F}B^{\prime\prime}=\angle CP_{F}B-\epsilon_{3},\quad
 \angle B^{\prime\prime}P_{F}A^{\prime\prime}=\angle BP_{F}A+\epsilon_{4},
\end{eqnarray*}
(where $\epsilon_{i}$  are non-negative real numbers for $i=3,4$)
such that
\[
 \angle A^{\prime\prime}P_{F}D^{\prime\prime}+\angle D^{\prime\prime}P_{F}C^{\prime\prime}+\angle C^{\prime\prime}P_{F}B^{\prime\prime}+\angle B^{\prime\prime}P_{F}A^{\prime\prime}=2\pi.
\]
The angular relations hold due to a result of A.D.~Alexandrov also
known as the angle comparison theorem (see \cite[Theorem 4,
p.~377, p.~54, case 1]{Alexandrov:06}).

Applying the same process that have been used in the proof of
Theorem~\ref{T-01}, we derive the desired plasticity equations
(\ref{plastic1M1}) and (\ref{plastic2M1}).
\end{proof}

\begin{remark}\rm
The comparative plasticity between the surfaces M and $M^{\prime}$
or $M^{\prime\prime}$ given in Corollary~\ref{prpr1} approaches
some type of formulation of the plasticity of a mathematical
"gibbosity" (part of a sphere on a convex surface) and it might be
considered as a useful tool for future medical applications in the
area of robotics.
\end{remark}
\section{A symmetrization of weighted quadrilaterals
on a surface of Gaussian curvature bounded above by a positive
number} \label{sec:4}

We introduce a new symmetrization technique which transforms some
classes of weighted convex quadrilaterals on a $C^{2}$ complete
convex surface of bounded Gaussian curvature $M$  to
parallelograms which lie on the same tangent plane that is defined
at the weighted Fermat-Torricelli point of the corresponding
quadrilateral.

We give two classes of parallelograms which characterize the
evolution of convex quadrilaterals which means that the weights
satisfy the plasticity equations of Theorem~\ref{T-01} \vskip1.5mm
The variable \textit{Gaussian curvature} ${K}$ is positive and
bounded below by  a real positive number $K_{1}:$ $K_{1}<{K}$,
then the geodesic distance of length not greater than
$\frac{\pi}{\sqrt{K_{1}}}$ and is unique, and an estimate of the
\textit{injectivity radius} $r_{i}=inf(r_{i}(P):P\in M)$ of $M$ is
given by the inequality $r_{i}\le\frac{\pi}{\sqrt{K_{1}}}\,$, see
\cite{VToponogov:09}, \cite[Theorem 3.5.2]{VToponogov:05}.

We consider a convex quadrilateral
$A^{\circ}B^{\circ}C^{\circ}D^{\circ}$ that belongs on a
neighborhood $W_{A}\subset M$, where $W_{A}$ is a subset of a
geodesic circle with center $A$ and radius $r_{i}$, and the
perimeter of the quadrilateral is smaller than $2\,r_{i}$.

Suppose that we select $w_R$ for
$R\in\{A^{\circ}B^{\circ}C^{\circ}D^{\circ}\}$
(Fig.~\ref{fig:tas1}) such that $P_{F}$ is located at the interior
domain of $A^{\circ}B^{\circ}C^{\circ}D^{\circ}$ and some
inequalities are also satisfied
\[w_{B^{\circ}}>w_{A^{\circ}}>w_{D^{\circ}}>w_{C^{\circ}}.\]

We consider a quadrilateral $A^{\circ}B^{\circ}C^{\circ}D^{\circ}$
that belongs on a neighborhood $W_{A}\subset M$, where $W_{A}$ is
a subset of a geodesic circle with center $A$ and radius $r_{i}$,
and the perimeter of the quadrilateral is smaller than $2\,r_{i}$.

Suppose that we select $w_R$ for
$R\in\{A^{\circ}B^{\circ}C^{\circ}D^{\circ}\}$
(Fig.~\ref{fig:tas1}) such that $P_{F}$ is located at the interior
domain of $A^{\circ}B^{\circ}C^{\circ}D^{\circ}.$

\begin{proposition}\label{theorsym}
A symmetrization of $A^{\circ}B^{\circ}C^{\circ}D^{\circ}$ with
respect to $P_{F}$ is the parallelogram
$A'^{\star}B'C'^{\star}D'$, where $A'^{\star},\,C'^{\star}$ are
the symmetric points of $A^{\prime},\,C^{\prime}$ with respect to
$P_{F}$ (see Figs.~\ref{fig:tas1} and \ref{fig:tas}),
\[
 |P_{F}R|=|\exp_{P_{F}}^{-1}(R)|=|P_{F}R^{\prime}|=w_{R}
\]
where $R$ lies on the geodesic $P_{F}R^{\circ}$ for
$R\in\{A,B,C,D\}$ and
$R^{\prime}\in\{A^{\prime},B^{\prime},C^{\prime},D^{\prime}\}.$

\end{proposition}

\begin{proof}[Proof of Proposition~\ref{theorsym}:]
The geodesic arcs $l_{R}$ for
$R\in\{A^{\circ},B^{\circ},C^{\circ},D^{\circ}\}$ belong to
$W_{A}$ and they are shortest arcs and any shortest arc is a
geodesic (see \cite[Lemma 3.5.2, p.~165 and Theorem 3.5.2,
p.~167]{VToponogov:05}). A particular case of the differentiation
of the length of a continuous family of shortest arcs given in
\cite[Corollary 4.5.7, p.~125]{Burago:01} is the differentiation
of the length of a $C^{2}$ family of geodesic arcs given in
\cite[Lemma~3.5.1, p.~164 and Remark 3.5.1]{VToponogov:05}.

 Let the geodesic arcs $l_{B^{\circ}}$,
$l_{C^{\circ}}$, $l_{D^{\circ}}$ be expressed as functions of
$l_{A^{\circ}}.$
\begin{equation} \label{eq:a_22gp}
 l_{B^{\circ}}=l_{B^{\circ}}(l_{A^{\circ}}),\quad
 l_{C^{\circ}}=l_{C^{\circ}}(l_{A^{\circ}}),\quad
 l_{D^{\circ}}=l_{D^{\circ}}(l_{A^{\circ}}).
\end{equation}
From (\ref{eq:a_22gp}) and (\ref{minimumcs}) the following
equation is obtained replacing $R\to R^{\circ},$ for
$R^{\circ}\in\{A^{\circ},B^{\circ},C^{\circ},D^{\circ}\}$
\begin{equation}\label{eq:B_2f_11gp}
 w_{A^{\circ}}l_{A^{\circ}}+w_{B^{\circ}}l_{B^{\circ}}(l_{A^{\circ}})+w_{C^{\circ}}l_{C^{\circ}}(l_{A^{\circ}})= \min
\end{equation}
Let $c(t):[0,T]\to M$ be a geodesic parameterized by arclength,
$Q$ a fixed point on $M$, $Q\ne c(0)$. Then the function $t \to
l_{Q}(t)=|Qc(t)|$ has the right derivative and
\begin{equation}\label{compute partialnngp}
\lim_{t\to
+0}\frac{l_{Q}(t)-l_{Q}(0)}{t}=\cos(\pi-\varphi_{c}^{\prime}(t)),
\end{equation}
where $\varphi_{c}^{\prime}(t)$ is the  angle between $c(t)$ and
the geodesic arc connecting $c(0)$ to $Q$.
\begin{figure}
\centering
\includegraphics[scale=0.30]{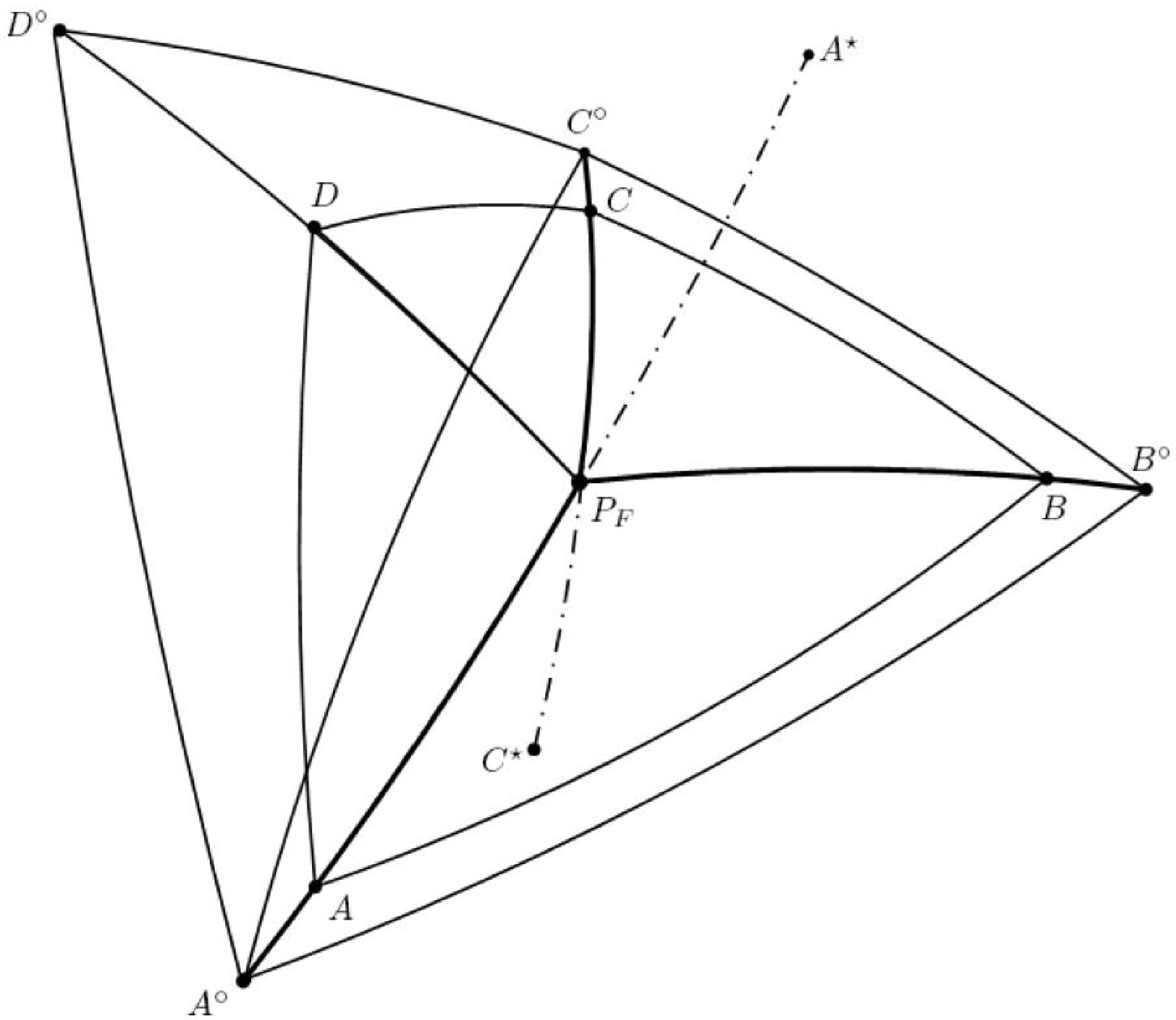}
\caption{} \label{fig:tas1}
\end{figure}

We choose the parametrization
\begin{equation}\label{compute partialparametertgp}
l_{A^{\circ}}(t)=t,
\end{equation}
Taking into account that $P_{F}$ is the w.\,F-T point of
$A^{\circ}B^{\circ}C^{\circ}D^{\circ}$, we derive the (right)
derivative of (\ref{eq:B_2f_11gp}) with respect to the variable
$l_{A^{\circ}}$ and considering (\ref{compute
partialparametertgp}), we get
\begin{eqnarray} \label{eq:B221fractgp}
 \nonumber
  w_{A^{\circ}}
 +w_{B^{\circ}}\lim_{t\to +0}\frac{l_{B^{\circ}}(t)-l_{B^{\circ}}(0)}{l_{A^{\circ}}(t)}
 +w_{C^{\circ}}\lim_{t\to+0}\frac{l_{C^{\circ}}(t)-l_{C^{\circ}}(0)}{l_{A^{\circ}}(t)}\\
 +w_{D^{\circ}}\lim_{t\to+0}\frac{l_{D^{\circ}}(t)-l_{D^{\circ}}(0)}{l_{A^{\circ}}(t)}=0.
\end{eqnarray}
From (\ref{compute partialnngp}) and (\ref{compute
partialparametertgp}), we get
\begin{equation}\label{partialimp2tgp}
 \lim_{t\to +0}\frac{l_{B^{\circ}}(t)-l_{B^{\circ}}(0)}{l_{A^{\circ}}(t)}
 =\cos(\angle A^{\circ}P_{F}B^{\circ}).
\end{equation}
Let the point $C^{\circ}$, the length parameterized curve $c(t)$
and the point $P_{F}$ that belongs in $c(t)$ be given in the
neighborhood $W_{P_{F}}$. Taking into account (\ref{compute
partialnngp}) and (\ref{compute partialparametertgp}), we have
(\cite[Lemma 3.5.1, p.~164 and Remark 3.5.1,
p.~165]{VToponogov:05})
\begin{equation}\label{partialimp3tgp}
 \lim_{t\to+0}\frac{l_{C^{\circ}}(t)-l_{C^{\circ}}(0)}{l_{A^{\circ}}}
 =\cos(\angle A^{\circ}P_{F}C^{\circ}).
\end{equation}
Similarly, let the point $D^{\circ}$, the length parameterized
curve $c(t)$ and the point $P_{F}$ that belongs in $c(t)$ be given
in the neighborhood $W_{A}$. Taking into account (\ref{compute
partialnngp}) and (\ref{compute partialparametertgp}), we obtain
\begin{equation}\label{partialimp4tbgp}
\lim_{t\to
+0}\frac{l_{D^{\circ}}(t)-l_{D^{\circ}}(0)}{l_{A^{\circ}}}=\cos(\angle
A^{\circ}P_{F}D^{\circ}).
\end{equation}
Replacing (\ref{partialimp2tgp}) --
(\ref{partialimp4tbgp}) in (\ref{eq:B221fractgp}), we obtain
\begin{equation}\label{eq:quadr115}
 w_{A^{\circ}}+w_{B^{\circ}}\cos(\angle A^{\circ}P_{F}B^{\circ})
 =-w_{C^{\circ}}\cos(\angle A^{\circ}P_{F}C^{\circ})-w_{D^{\circ}}\cos(\angle A^{\circ}P_{F}D^{\circ}).
\end{equation}
Similarly, working cyclically and differentiating
(\ref{minimumcs}) with respect to $l_{B^{\circ}}$, and choosing
the parametrization $l_{B^{\circ}}(t^{\prime})=t^{\prime}$,
differentiating (\ref{minimumcs}) with respect to $l_{C^{\circ}}$
and choosing the parametrization
$l_{C^{\circ}}(t^{\prime\prime})=t^{\prime\prime}$ and
differentiating (\ref{minimumcs}) with respect to $l_{D^{\circ}}$
and choosing the parametrization
$l_{D^{\circ}}(t^{\prime\prime\prime})=t^{\prime\prime\prime}$, we
derive three relations, respectively,
\begin{eqnarray}\label{equation2tgp}
 w_{A^{\circ}}\cos(\angle
 B^{\circ}P_{F}A^{\circ})+w_{B^{\circ}}=-w_{C^{\circ}}\cos(\angle
 B^{\circ}P_{F}C^{\circ})-w_{D^{\circ}}\cos(\angle
 B^{\circ}P_{F}D^{\circ}),\\
\label{equation3tgp}
 w_{A^{\circ}}\cos(\angle
 C^{\circ}P_{F}A^{\circ})+w_{B^{\circ}}\cos(\angle
 C^{\circ}P_{F}B^{\circ})=-w_{C^{\circ}}-w_{D^{\circ}}\cos(\angle
 C^{\circ}P_{F}D^{\circ}),\\
\label{equation4tbgp}
 w_{A^{\circ}}\cos(\angle
 D^{\circ}P_{F}A^{\circ})+w_{B^{\circ}}\cos(\angle
 D^{\circ}P_{F}B^{\circ})=-w_{C^{\circ}}\cos(\angle
 D^{\circ}P_{F}C^{\circ})-w_{D^{\circ}}.
\end{eqnarray}
The equations (\ref{eq:quadr115})\,--\,(\ref{equation4tbgp}) could
be written in a unified form
\[
 \sum_{Q\in\{A^{\circ},\,B^{\circ},\,C^{\circ},\,D^{\circ}\}}w_{Q}\cos(\angle QP_{F}R)=0\quad
 \mbox{for}\quad R\in\{A^{\circ},B^{\circ},C^{\circ},D^{\circ}\}.
\]
Consider the 4-inverse w.\,F-T problem under the condition
\[
 \sum_{R\in\{A^{\circ},\,B^{\circ},\,C^{\circ},\,D^{\circ}\}}(w_{R})_{A^{\circ}B^{\circ}C^{\circ}D^{\circ}}=
 c.
\]
Solving the linear system
(\ref{eq:quadr115})\,--\,(\ref{equation2tgp}) with respect to
$w_{A^{\circ}}$ and $w_{B^{\circ}}$, we obtain
\begin{eqnarray}\label{ratio1tgp}
 w_{A^{\circ}}\sin(\angle B^{\circ}P_{F}A^{\circ})+w_{C^{\circ}}\sin(\angle
 B^{\circ}P_{F}C^{\circ})+w_{D^{\circ}}\sin(\angle B^{\circ}P_{F}D^{\circ})=0,\\
 \label{ratio2tgp}
 w_{B^{\circ}}\sin(\angle A^{\circ}P_{F}B^{\circ})+w_{C^{\circ}}\sin(\angle
 A^{\circ}P_{F}C^{\circ})+w_{D^{\circ}}\sin(\angle A^{\circ}P_{F}D^{\circ})=0.
\end{eqnarray}
 Similarly, solving the system (\ref{eq:quadr115}) and (\ref{equation3tgp}) with respect to
 $w_{A^{\circ}},$ $w_{C^{\circ}}$, we obtain
\begin{equation}\label{ratio3tgp}
 w_{A^{\circ}}\sin(\angle C^{\circ}P_{F}A^{\circ})+w_{B^{\circ}}\sin(\angle
 C^{\circ}P_{F}B^{\circ})+w_{D^{\circ}}\sin(\angle C^{\circ}P_{F}D^{\circ})=0.
\end{equation}
We write (\ref{ratio2tgp}) in the following form:
\begin{equation} \label{eq:quadr15Kquad}
w_{B^{\circ}}\sin(\angle
A^{\circ}P_{F}B^{\circ})=-w_{C^{\circ}}\sin(\angle
A^{\circ}P_{F}C^{\circ})-w_{D^{\circ}}\sin(\angle
A^{\circ}P_{F}D^{\circ}).
\end{equation}
From the derived equations (\ref{eq:quadr115}) --
(\ref{ratio3tgp}) we obtain a balancing condition of four tangent
vectors at $P_{F}$ which are located at the tangent plane
$T_{P_{F}}(M)$ at $P_{F}$ having their weighted sum zero. By this
approach, we deduce the invariance property of the w.\,F-T point
$P_{F}$ for a given convex quadrilateral
$A^{\circ}B^{\circ}C^{\circ}D^{\circ}$ which states that:

\vskip1mm Suppose that there is a convex quadrilateral
$A^{\circ}B^{\circ}C^{\circ}D^{\circ}$
 on $M$ and a non-negative weight $w_{R}$ corresponds at each vertex $R,$ for $R\in\{A^{\circ}B^{\circ}C^{\circ}D^{\circ}\}$. Assume that the
 w.\,F-T point $P_{F}$ point is an interior point of
$A^{\circ}B^{\circ}C^{\circ}D^{\circ}.$ If $P_{F}$ is connected
with every vertex $R$ and a point $S$ is selected with a
non-negative weight $w_{R}$ such that $S$ belongs to the geodesic
that is defined by the geodesic arc $P_{F}R,$ for
$R\in\{A^{\circ},B^{\circ},C^{\circ},D^{\circ}\}$ and
$S\in\{A,B,C,D\}$ and the convex quadrilateral $ABCD$ is
constructed such that the corresponding weighted Fermat-Torricelli
point $P_{{F}}'$ is not a vertex of $ABCD,$ then the w.\,F-T point
$P_{F}'$ is identical with $P_{F}$.
 Squaring both parts of (\ref{eq:quadr115}) and
(\ref{eq:quadr15Kquad}) and adding the two derived equations, we
get
\begin{equation}\label{eq:c1}
w_{A^{\circ}}^2+w_{B^{\circ}}^2+2w_{A^{\circ}}w_{B^{\circ}}\cos(\angle
A^{\circ}P_{F}B^{\circ})=w_{C^{\circ}}^2+w_{D^{\circ}}^2+2w_{C^{\circ}}w_{D^{\circ}}\cos({\angle
C^{\circ}P_{F}D^{\circ}}).
\end{equation}
\begin{figure}
\centering
\includegraphics[scale=0.35]{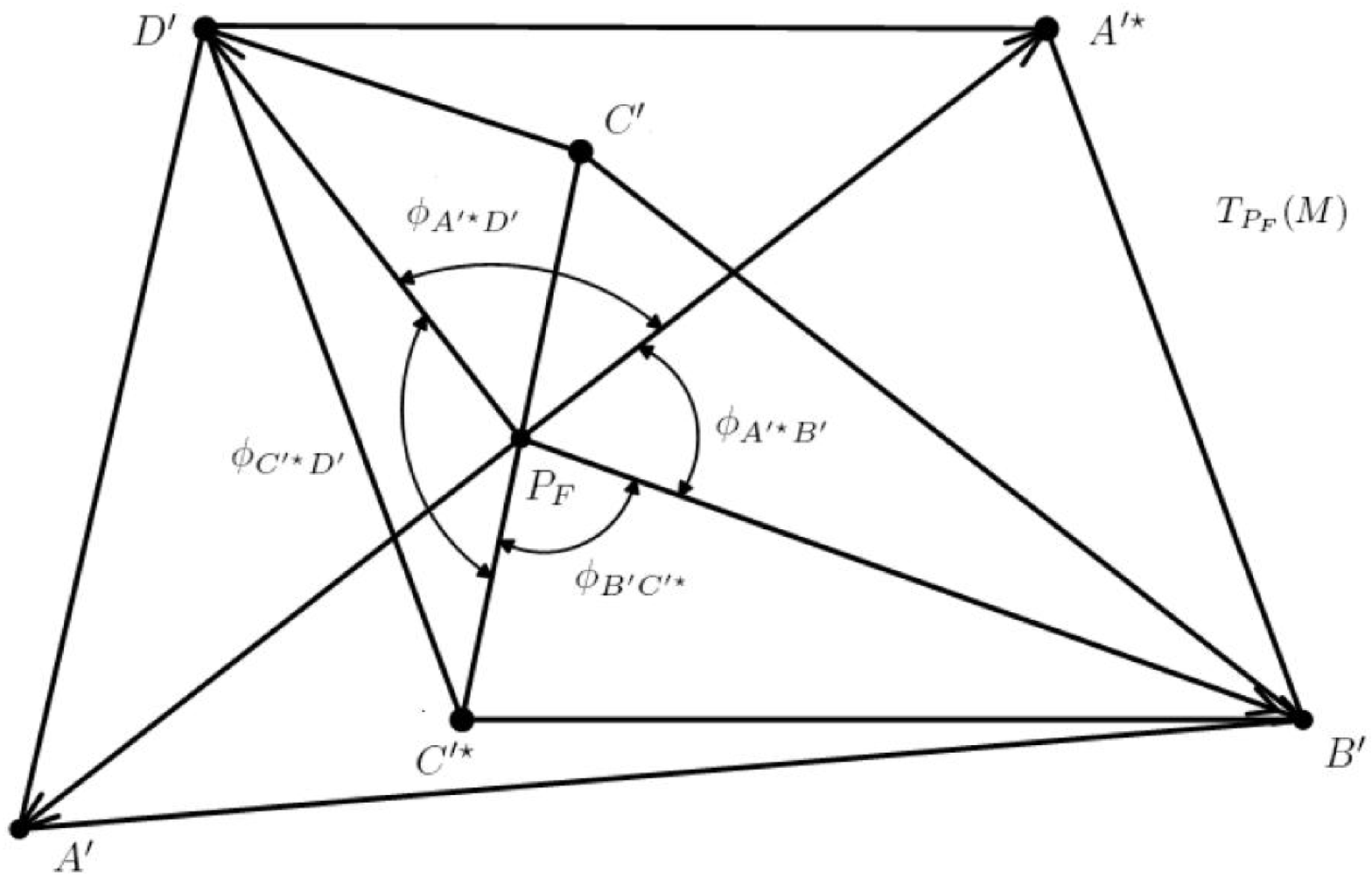}
\caption{} \label{fig:tas}
\end{figure}
Applying the same process and exchanging the indices
$D^{\circ}\leftrightarrow B^{\circ}$ in (\ref{eq:c1}), we get
\begin{equation}\label{eq:c2}
w_{A^{\circ}}^2+w_{D^{\circ}}^2+2w_{A^{\circ}}w_{D^{\circ}}\cos(\angle
A^{\circ}P_{F}D^{\circ})=w_{B^{\circ}}^2+w_{C^{\circ}}^2+2w_{B^{\circ}}w_{C^{\circ}}\cos(\angle
B^{\circ}P_{F}C^{\circ}).
\end{equation}
The invariance property of the w.\,F-T $P_{F}$ gives us the
ability to transform the initial quadrilateral
$A^{\circ}B^{\circ}C^{\circ}D^{\circ}$ to $ABCD$ and applying the
inverse of the exponential mapping with respect to $P_{F}$, we can
move to $A'B'C'D'$ (see Figs.~\ref{fig:tas1},~\ref{fig:tas}) such
that
\[
 |P_{F}R|=|\exp_{P_{F}}^{-1}(R)|=|P_{F}R^{\prime}|=w_{R},
\]
where $R$ belongs to the geodesic $P_{F}R^{\circ}$, for
$R\in\{A,B,C,D\}$,
$R^{\prime}\in\{A^{\prime},B^{\prime},C^{\prime},D^{\prime}\}.$
Furthermore, we construct the symmetric points of $A$, $C$ with
respect to $P_F$, $A^{\star}$ and $C^{\star}$, respectively, such
that $l_{A^{\star}}(P_{F})=l_{A}(P_{F})$ and
$l_{C^{\star}}(P_{F})=l_{C}(P_{F})$ on $T_{P_{F}}(M)$.
Similarly,we construct the symmetric points of $A^{\prime}$,
$C^{\prime}$ with respect to $P_F$, $A^{\prime\star}$ and
$C^{\prime \star}$, respectively, such that
$l_{A^{\prime\star}}(P_{F})=l_{A^{\prime}}(P_{F})$ and
$l_{C^{\prime\star}}(P_{F})=l_{C^{\prime}}(P_{F})$ on
$T_{P_{F}}(M)$ (see Figs.~\ref{fig:tas1} and \ref{fig:tas}).
 Taking into consideration Fig.~\ref{fig:tas}, we deduce
\begin{eqnarray}\label{eq:c3}
\nonumber
 &&\phi_{B^{\prime}C^{\prime \star}}=\pi-\angle B^{\circ}P_{F}C^{\circ},\quad
 \phi_{A^{\prime\star}D^{\prime}}=\pi-\angle A^{\circ}P_{F}D^{\circ},\\
 &&\phi_{C^{\prime \star}D^{\prime}}=\pi-\angle C^{\circ}P_{F}D^{\circ},\quad
 \phi_{A^{\prime \star}B^{\prime}}=\pi-\angle A^{\circ}P_{F}B^{\circ}.
\end{eqnarray}
Replacing (\ref{eq:c3}) in (\ref{eq:c1}), (\ref{eq:c2}), we get
$l_{A'^{\star}}(B')=l_{C'^{\star}}(D')$,
$l_{A'^{\star}}(D)'=l_{B'}(C'^{\star})$. Therefore,
$A'^{\star}B'C'^{\star}D'$ is a parallelogram on $T_{P_{F}}(M).$
\end{proof}

\begin{proposition}\label{P-04}
If $w_{A^{\circ}}=w_{C^{\circ}}$ and
$w_{B^{\circ}}=w_{D^{\circ}},$ then
$A^{\circ}B^{\circ}C^{\circ}D^{\circ}$ is transformed directly
with respect to $P_{F}$ to a parallelogram $A'B'C'D'.$
\end{proposition}

\begin{proof}[Proof of Proposition~\ref{P-04}:]
Applying the invariance property (geometric plasticity) of $P_{F}$
we transform the initial quadrilateral
$A^{\circ}B^{\circ}C^{\circ}D^{\circ}$ to $ABCD$ and applying the
inverse of the exponential mapping with respect to $P_{F}$, we get
$A'B'C'D'$ such that
\[|P_{F}R|=|\exp_{P_{F}}^{-1}(R)|=|P_{F}R'|=w_{R},\]
for $R\in\{A,B,C,D\}$ and
$R^{\prime}\in\{A^{\prime},B^{\prime},C^{\prime},D^{\prime}\}$. We
conclude that $A'B'C'D'$ is a parallelogram because the diagonals
bisect.
\end{proof}

We proceed by proving the following theorem:

\begin{theorem}\label{T-04}
The geometrization of the plasticity of convex quadrilaterals on
$M$ is given by the following two classes of parallelograms:

\underline{Class A}: Parallelograms derived from the initial
quadrilateral $A^{\circ}B^{\circ}C^{\circ}D^{\circ}$ to the
tangent plane at $P_{F}$ by taking the symmetric points of
$A^{\prime}$ and $C^{\prime}$ with respect to $P_{F}$ such that
\[
 |P_{F}R|=|\exp_{P_{F}}^{-1}(R)|=|P_{F}R'|=w_{R},
\]
where $R$ lies on the geodesic $P_{F}R^{\circ}$ for
 $R\in\{A,B,C,D\}$ and $R^{\prime}\in\{A^{\prime},B^{\prime},C^{\prime},D^{\prime}\}.$

\underline{Class B}: Parallelograms derived from the initial
quadrilateral to the tangent plane at $P_{F}$ by taking the
symmetric points of $B^{\prime}$ and $D^{\prime}$ with respect to
$P_{F}.$
\end{theorem}

\begin{proof}[Proof of Theorem~\ref{T-04}:]
We consider the following two cases:

\vskip1mm 1.\quad
$w_{B^{\circ}}>w_{A^{\circ}}>w_{D^{\circ}}>w_{C^{\circ}}.$

\vskip1mm\noindent Applying Theorem~\ref{theorsym}, we derive
class A of Theorem~\ref{T-04}.

\vskip1mm 2.\quad
$w_{A^{\circ}}>w_{B^{\circ}}>w_{C^{\circ}}>w_{D^{\circ}}.$

\vskip1mm \noindent Setting $A^{\circ}\to B^{\circ}$,
$B^{\circ}\to A^{\circ}$, $C^{\circ}\to D^{\circ}$, $D^{\circ} \to
C^{\circ},$ and applying Theorem~\ref{theorsym}, we derive class B
of Theorem~\ref{T-04}.

We note that class A and class B contain similar and non similar
parallelograms. Similar parallelograms are derived by the
condition of the 4-inverse w.\,F-T problem:
\[
 w_{A^{\circ}}+w_{B^{\circ}}+w_{C^{\circ}}+w_{D^{\circ}}=const,
\]
by scaling the constant which takes positive real values.
\end{proof}

\begin{example}\label{verfour}\rm
We consider for simplicity a convex quadrilateral
$A^{\circ}B^{\circ}C^{\circ}D^{\circ}$ in $\mathbb{R}^{2},$ with
given distances, angles  and weights
\begin{eqnarray*}
  l_{A^{\circ}}\eq 5,\quad l_{B^{\circ}}=7.5,\quad l_{C^{\circ}}=5,\quad l_{D^{\circ}}=10,\\
  \angle A^{\circ}P_{F}B^{\circ}\eq 120^\circ,\
 \angle B^{\circ}P_{F}C^{\circ}=90^\circ,\
 \angle C^{\circ}P_{F}D^{\circ}=50^\circ,\
 \angle D^{\circ}P_{F}A^{\circ}=100^\circ,\\
  w_{A^{\circ}}\eq 0.81,\quad w_{B^{\circ}}=0.712,\quad w_{C^{\circ}}=0.444,\quad w_{D^{\circ}}=0.4
\end{eqnarray*}
taken from \cite[Example 4.7, p.~418]{Zachos/Zou:88},
 for the case of $\mathbb{R}^{2}$, where $P_F$ is the
corresponding weighted Fermat-Torricelli point which is the
intersection of the four prescribed lines $P_{F}R$ for
$R\in\{A^{\circ},B^{\circ},C^{\circ},D^{\circ}\}$ (see
Fig.~\ref{fig:tas56}). The convex quadrilateral
$A^{\circ}B^{\circ}C^{\circ}D^{\circ}$ of Fig.~\ref{fig:tas55} has
the same angles $\angle RP_{F}S,$ and lengths $l_{R}$, for
$R,S\in\{A^{\circ},B^{\circ},C^{\circ},D^{\circ}\}$, like in
Fig.~\ref{fig:tas56} with weights $w_{A^{\circ}}=0.76,
w_{B^{\circ}}=0.76, w_{C^{\circ}}=0.34$, $w_{D^{\circ}}=0.5$ taken
from \cite[Example 4.7]{Zachos/Zou:88}.
 The plasticity equations of Theorem~\ref{T-01} and Corollary~\ref{corol} show
that the w.\,F-T point $P_F$ remains the same for
Figures~\ref{fig:tas56} and~\ref{fig:tas55}. We notice that an
increase of $w_{D^{\circ}}$ from 0.4 to 0.5 causes a decrease to
$w_{A^{\circ}},$ $w_{C^{\circ}}$ and an increase to
$w_{B^{\circ}}$ (plasticity). The weights that correspond to
Fig.~\ref{fig:tas56} derive a parallelogram
$A'^{\star}B'C'^{\star}D'$ that belongs to class A and the weights
that correspond to Fig.~\ref{fig:tas55} derive a parallelogram
$A'B'^{\star}C'D'^{\star}$ that belongs to class B. The weights of
Figures~\ref{fig:tas56} and~\ref{fig:tas55} satisfy the plasticity
equations for $const=2.37$.
\end{example}

\begin{figure}
\centering
\includegraphics[scale=0.40]{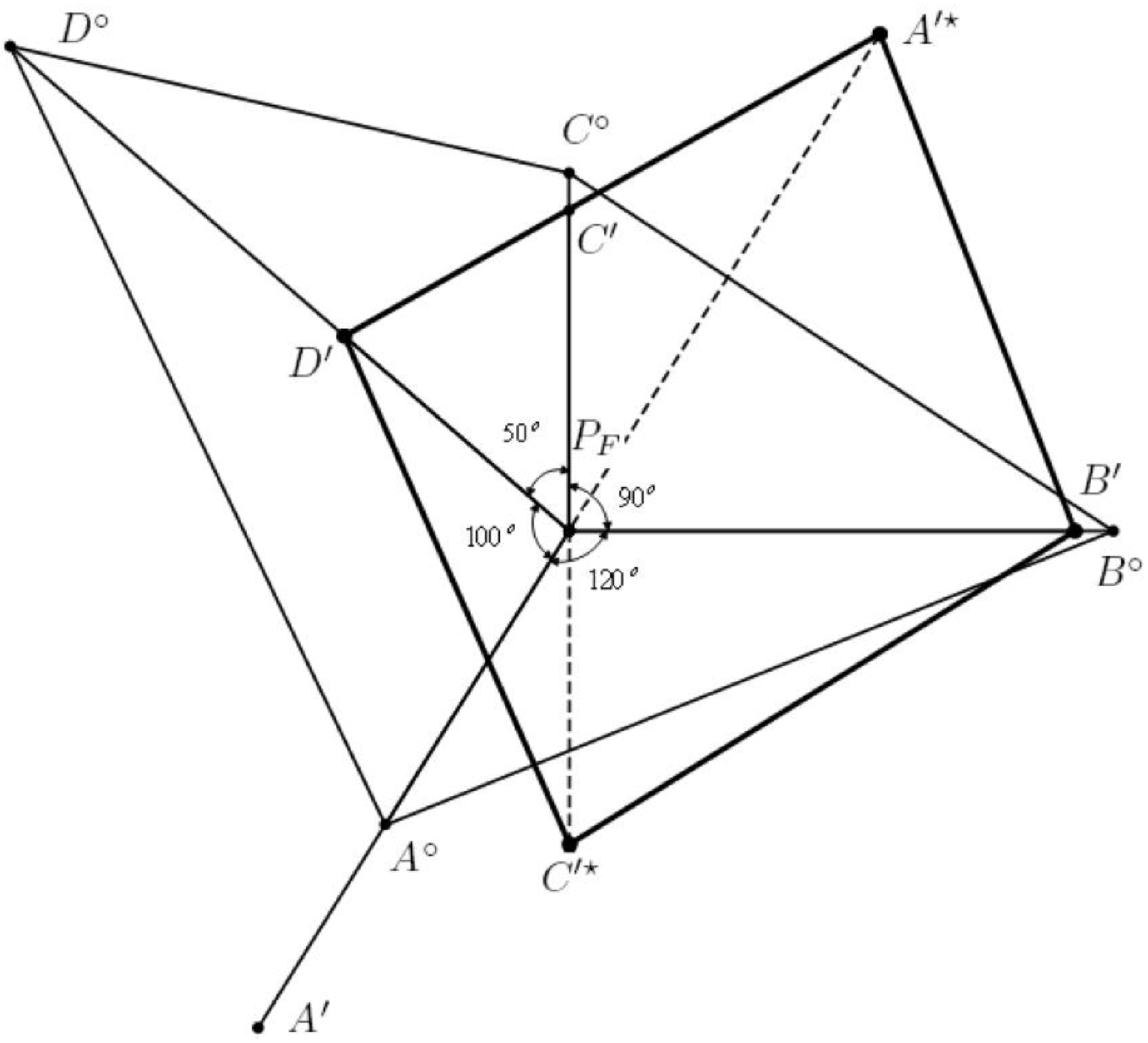}
\caption{} \label{fig:tas56}
\end{figure}

\begin{figure}
\centering
\includegraphics[scale=0.40]{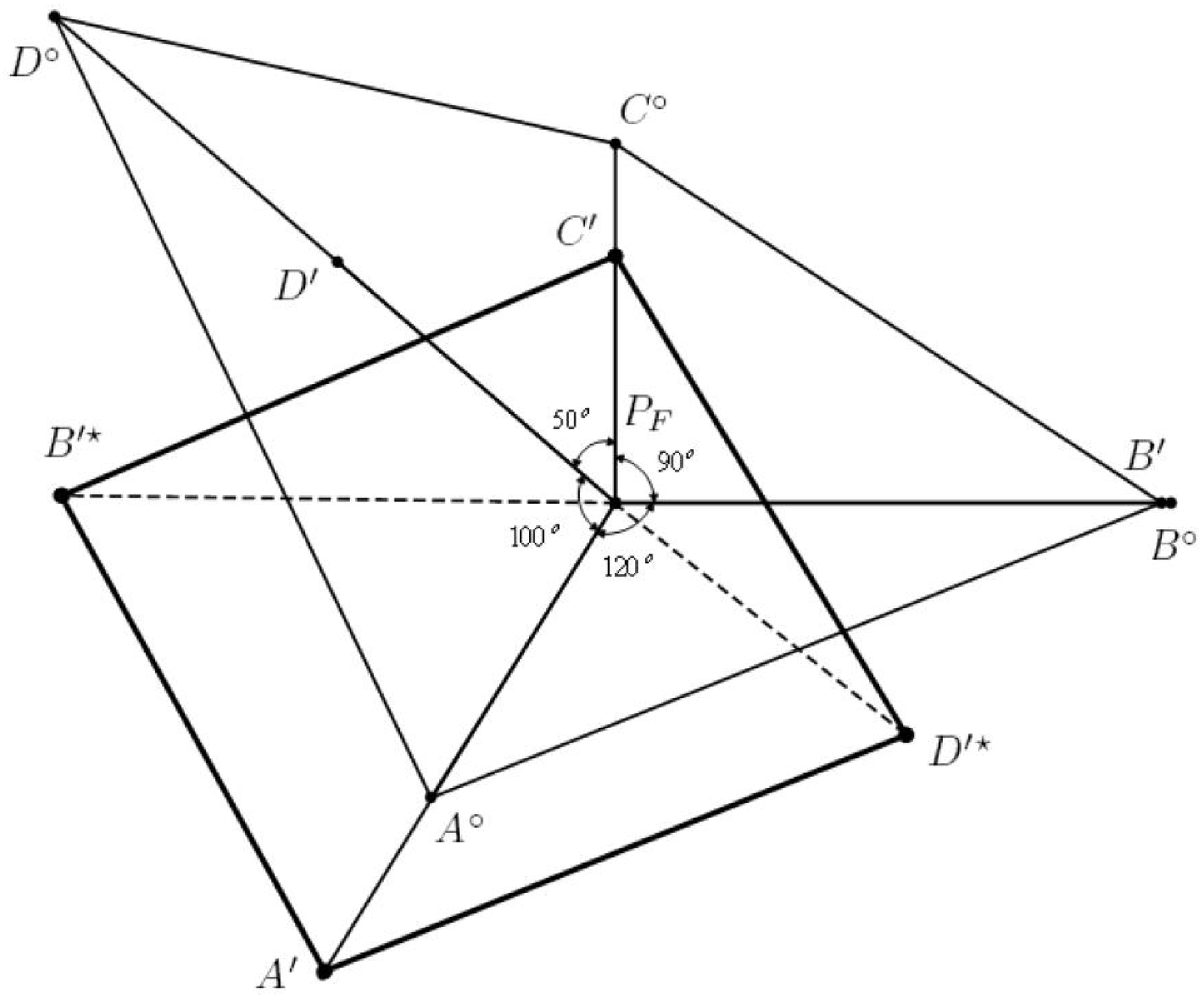}
\caption{} \label{fig:tas55}
\end{figure}

\begin{remark}\rm
We have excluded some trivial cases where the weighted
quadrilaterals can be transformed directly to parallelograms and
we consider quadrilaterals that the initial starting values of
their weights which satisfy the plasticity equations must also
satisfy the inequalities
$w_{B^{\circ}}>w_{A^{\circ}}>w_{D^{\circ}}>w_{C^{\circ}}$ or
$w_{A^{\circ}}>w_{B^{\circ}}>w_{C^{\circ}}>w_{D^{\circ}}.$
\end{remark}


\vskip1mm The author is grateful to Professor Dr.Vladimir Rovenski
for many useful discussions, comments on the paper and for
communicating to him remark 2.


\end{document}